\documentclass{aptpub}
\usepackage{amssymb}
\usepackage{amsmath}
\usepackage{amsfonts}
\usepackage{accents} 
\usepackage{pgf}
\usepackage{tikz}
\usepackage{mathtools}
\usepackage{fancyhdr, layout}
\usepackage{changepage}
\usepackage{hyperref}
\usetikzlibrary{shadows}
\usepackage[T1]{fontenc}
\usepackage{xspace}

\graphicspath{{/Users/laurenssmit/Documents/Promoveren/Pictures/}}
\newcommand{\uG}[1]{B^{#1}}

\newcommand{\lpc}{lattice path counting}
\newcommand{\philr}[1]{\phi\langle #1 \rangle}

\newcommand{\stQ}{\widehat{Q}}
\newcommand{\stR}{\widehat{R}}
\newcommand{\stq}{\hat{q}}
\newcommand{\str}{\hat{r}}
\newcommand{\ie}{i.e.,~}
\newcommand{\LPC}{lattice path counting}
\newcommand{\eqqref}[1]{Eq.~(\ref{#1})}

\newcommand{\qbdp}{\ensuremath{\text{QDESA}^+}\xspace}
\newcommand{\qbdpp}{\ensuremath{\text{QDESA}^{++}}\xspace}

\newcommand{\ssp}{\ensuremath{\mathcal{X}}}

\newcounter{Alg}
\authornames{KATEHAKIS ET AL.} 
\shorttitle{LPCA and QDESA} 
 
\begin{document}

\title{A Comparative Analysis  of  the Successive Lumping and the Lattice Path Counting Algorithms} 

\authorone[Rutgers University,  NJ, USA]{Michael N. Katehakis} 
\addressone{Department of Management Science and Information Systems,
Rutgers University,
 100 Rockafeller Road, Piscataway, NJ 08854, USA.
 E-mail: mnk@rutgers.edu} 
\vspace{-1cm}
 \authortwo[Leiden University, Netherlands]{Laurens C. Smit}
 \addresstwo{Mathematisch Instituut, Universiteit Leiden,
Niels Bohrweg 1, 2333 CA, The Netherlands,
 E-mail: lsmit@math.leidenuniv.nl}
\vspace{-1cm}
 \authorthree[Leiden University, Netherlands]{Floske M. Spieksma}
 \addressthree{Mathematisch Instituut, Universiteit Leiden,
Niels Bohrweg 1, 2333 CA, The Netherlands,
 E-mail: spieksma@math.leidenuniv.nl}

\begin{abstract}
{\small 
This article provides 
a comparison of the   successive lumping (SL) methodology developed in    \cite{sslqsf2015} with the  popular  \LPC\ \cite{mohanty1979lattice} in 
 obtaining  rate matrices for  queueing models, satisfying the  specific  quasi birth and death structure as in  \cite{leeuwaarden2009quasi}, \cite{van2006quasi}. The two methodologies are compared both in terms 
 of applicability requirements and numerical complexity
 by  analyzing  their  performance   for the same classical  queueing models considered in 
 \cite{leeuwaarden2009quasi}. 
 The main findings are: i) When both methods 
are applicable the SL based algorithms outperform the 
  \LPC \  algorithm (LPCA).  ii) There are important classes of problems (e.g., models with  (level) non-homogenous rates or with finite state spaces) for which the SL  methodology 
is applicable and for which the LPCA cannot be used. iii) Another   
  main advantage of successive lumping algorithms over \LPC\ is that the former includes
   a  method to compute  the steady state distribution using this rate matrix. 
  }
\end{abstract}



\section{Introduction}

Two dimensional Markov chains arise as a natural way to model various real life applications.
In particular, many queueing models possess this structure and it is even possible that a more complex, higher dimensional queueing model can be decomposed into various two dimensional Markov processes. 
For various queueing models we refer to \cite{adan2009synchronized,adan2015shorter,adan2013erlang,bohm1997combinatorics, etessami2010quasi, 
Armin2003,Perros94,L-2011admission, vlasiou2014insensitivity,zhao1995queueing}.
Other areas in which these processes will arise outside queueing are for example  inventory models, cf. 
\cite{qr2012}, reliability, cf. \cite{dra1988,m1989} and 
 pricing models.  
In this paper we are particularly interested in 
a comparison of the new successive lumping (SL) methodology developed in  \cite{sslqsf2015} with the  popular  \LPC\ \cite{mohanty1979lattice} in 
 obtaining  rate matrices for  queueing models, as in  \cite{van2006quasi} and \cite{leeuwaarden2009quasi}. The two methodologies are compared both in terms 
 of applicability requirements and numerical complexity
 by  analyzing  their  performance   for the same classical  queueing models considered in 
 \cite{leeuwaarden2009quasi}.  
In all these models, the  objective is  to calculate the steady state distribution of a  pertinent   Quasi Birth-and-Death (QBD) process (i.e., a two dimensional Markov chain with a 
  transition generator  matrix $Q$ that contains nonzero rates  only for  transitions to the `left' and to the 
  `right' in every state) that describes the evolution of the state of the system in time.

The main  method that is  used to analyze QBD processes is based on  expressing the stationary probabilities  of states  of one  level in terms of those of its previous levels. This is done with the aid of a rate matrix $R$, which is 
the basis  of the matrix-geometric solution introduced by Neuts.
For general level-independent QBD processes, it is known that $R$ satisfies a matrix-quadratic equation. Algorithms for solving this equation were given in \cite{neuts80} and Latouche and Ramaswami \cite{latouche1999introduction}.
A current state of the art software implementing quadratically-convergent algorithms with a number of speed-up features is described in \cite{bmsv06}. A general algorithm for the level-independent case can be found in \cite{bright1995calculating} and a discussion of the Quasi Skip Free case in \cite{liu1996determination}.

There are various methods that make use of a special structure of the transition rate matrix $Q$, to provide 
efficient computation procedures for the rate matrix $R$. Such a procedure is available in the 
 the case in which the `down matrix' of $Q$,   is a product of a row and a column vector.  
 For other procedures that explicitly calculate a rate matrix we refer to \cite{van2011triangular} and \cite{mohanty1990discrete}. 
Recent studies, cf.~\cite{van2006quasi, leeuwaarden2009quasi}, have used \lpc\  methods to directly compute  the rate matrix for certain QBD processes that arise in queueing models.  For example, a priority queue model has been analyzed by this method, but also with other techniques, see e.g.~\cite{gillent1983semi} and references therein.
The idea of counting the number of paths on a lattice, cf.~\cite{mohanty1979lattice, flajolet2000formal},   has been used in many fields of applied probability, cf. \cite{spitzer2001principles}.

 A new  alternative method to  compute  the rate matrix for certain QBD processes can be  based on 
  the successive lumping (SL) procedure introduced in \cite{ssl2012}. 
It was  employed   in \cite{sslqsf2015} to obtain explicit solutions for `rate sets'  for large classes  of    QSF processes,  the so-called DES and RES processes.  The SL   approach   differs from  the previous mentioned works by its distinct method of derivation and its  applicability  to models with  infinite state spaces  and models that are   outside the QSF framework. 
However, it should be noted  that algorithms given in 
 \cite{gillent1983semi, latouche1999introduction, bright1995calculating}   can be used on other, more general (in terms of down-transitions) processes. The advantages of using SL are described in \cite{sslqsf2015}. 
Although the nature of a  path counting based method and the successive lumping based method 
   are very different, a comparison can be done, since they both rely on the absence of certain kind of transitions. 
Herein we compare the method introduced in \cite{leeuwaarden2009quasi} with the one based on successive lumping of \cite{sslqsf2015}.  

The main contribution of this paper is to  provide a clear comparison between  successive lumping (SL) based methods and the lattice path counting based algorithm, introduced in \cite{leeuwaarden2009quasi},   in computational complexity and applicability. First, it is shown that  the SL methodology yields  algorithms that are   faster than the counting algorithm.
Second, we show that SL based procedures are applicable to many  of the queueing models discussed in previous papers, and even to
  models   with finite state spaces  or   with   non-homogenous transition rate structures and to models with a quasi skip free (QSF) structure, cf.  \cite{sslqsf2015}.   However, there seem to exist  some artificial   queueing models  that do not possess  the SL property,   for which a lattice path counting algorithm is applicable.  Finally, this paper  continues the work of   \cite{sslqsf2015},  and it specializes its results  to homogenous QBD processes, in order to make the  comparison  of  successive lumping (SL) based  methods and the \lpc \  procedure possible. 

The paper has the following structure. In Section \ref{sec:def} we first  define the notation for the QBD processes that we will use throughout the paper.  In Section~\ref{sec:def} we  summarize the results of \cite{sslqsf2015} for the DES processes as they apply to quasi birth and death processes with a down entrance state and the resulting  
  {\sl  quasi birth and death down entrance state algorithm} (QDESA).
 In Section~\ref{sec:def} the QDESA procedure is specialized depending on the structure of the transition rate 
 $Q$, applicable  to the models under investigation in this paper.  Then, in Section~\ref{sec:examples} the introduced procedures are clarified by applying 
 them to two specific queueing examples.
 In Section~\ref{sec:lpc} we review the lattice path counting algorithm. In Section~\ref{sec:comp} we compare the  procedures  in speed (computational complexity). In   Section~\ref{sec:app} we discuss the  
  the type  of models for which each procedure can be applied. We
   conclude with some models that further   illustrate these comparisons. 

\section{Preliminary Results}\label{sec:def}
\subsection{Successive Lumping in Quasi Birth and Death Processes}\label{sec:sl}
In the sequel we consider an ergodic QBD process $X(t)$  with states in a finite or countable set $\ssp$. The 
states  (after re-labeling) will be written as tuples 
$(m,i)$, where in the state description the first entry  $m=0,1,\ldots,M$ represents the `level' of the state and
the second entry $i=0,1,2,\ldots,\ell_{m} $ represents the `stage' of the state $(m,i)$. 
 The integers  $\ell_{m}$ and $M$ are given constants  and they represent respectively  the number of stages ($\ell_{m}+1 $) and the highest level ($M $);  these scalars can be infinite. %
Let $Q$ denote the transition generator  matrix. 
The process $X(t)$ is  referred to as a `level QBD' process if  the only transitions  allowed are to a 
 state that is within the same  level  or to a level one step above or below, \ie$Q$
has the form:
\begin{equation}\label{eq:q}
Q=\left[\begin{array}{cccccc} W^{0}&U^0&0&\cdots& 0&0\\
D^1&W^1 &U^1 &\ddots& 0&0\\
0&D^2 &W^2 &\ddots& 0&0\\
\vdots&\ddots&\ddots&\ddots&\vdots&\vdots\\
0&0&0&\cdots&W^{M-1}& U^{M-1}\\
0&0&0&\cdots&D^M& W^M
\end{array}\right].
\end{equation}
The matrices $W$, $D$ and $U$ represent `within a level', `down one level' and `up one level' transitions respectively. The sub-matrices $W^{m}$
 above are of  dimension  $(\ell_{m}+1) \times (\ell_{m}+1)$, the sub-matrices $D^{m}$ are of dimension $(\ell_{m}+1)\times (\ell_{m-1}+1)$ and the submatrices $U^{m}$ are of dimension $(\ell_{m}+1)\times (\ell_{m+1}+1).$
 Further, we will use the notation $\mathcal{L}_n=\{(n,i),\ i=0,1,\ldots,\ell \}$  for the level sets ($n=0,1,\ldots,M$).


Let $\pi$ denote   the steady state distribution, \ie the solution of $\pi Q=0$  and $\pi1=1$. We denote by  $\pi^{n}$ the sub-vector of $\pi$ formed by the stationary probabilities of  the states of level $n$  \ie$\pi^{n}=[\,\pi(n,0),\ldots,\pi(n,\ell)\,]. $ 

In the context of the current paper we will
assume that every matrix $D^{m}$ has only one nonzero column (that for this section we will assume be the first column). The underlying QBD process is therefore successively lumpable (a DES process) with 
respect to the partition $\{\mathcal{L}_n\}_{n\ge0}$ of the state space $\ssp$,  cf.~\cite{ssl2012} for lumping and \cite{sslqsf2015} for a proof that $X(t)$ is lumpable with respect to this partition.
In addition we will assume that $\ell_{m}=\ell$  for all $m$ (i.e., the level size is independent of the level) and note that this condition is not   necessary for the DES procedure to be applicable, but is necessary for the LPC procedure, that will be discussed in Section~\ref{sec:lpc}. Below we will repeat the important definitions from \cite{sslqsf2015}, specialized for a QBD process.

In a QBD process we define the matrix $\widetilde{U}^m$  of size $(\ell+1)\times (\ell+1)$ as follows:
\begin{equation}\label{Eq:qbdU}
\widetilde{U}^m= U^m 1_{m}'\delta_m,
\end{equation}
where $1_{m}$ is a rowvector of size $\ell+1$ with identically equal to $1$ and $\delta_{m}$ is a vector of the same size identically equal to $0$ with a $1$ on its first entry.
Furthermore we define:
\begin{equation}\label{Eq:qbdB}
B^{m}=W^{m}+\widetilde{U}^{m}.
\end{equation}
For a  QBD process, we will call a matrix set $\{\mathcal{R}_{m}\}_{m}$ that satisfies the equation below  
\emph{a rate matrix set}.
\begin{equation}\label{Eq:qbdR}
\pi^{m}=\pi^{m-1}\mathcal{R}_m\, , \ \ \ \ \mbox{for $m=1,\ldots,M_2$}.
\end{equation}

In  \cite{sslqsf2015} it was shown that  the   matrix  $\uG{m}$ is invertible. A simplification of Theorem 2 of that paper 
for the special case of a QBD process implies that 
the matrix set  $\mathcal{R}_0:=\{R_{m}\}_{m}$  defined by:  
\begin{equation}\label{Eq:qbd}
 R_m=- U^{m-1} (B^m)^{-1},
\end{equation}
is a rate matrix set for  $Q$, when  $D^{m}$ has a single nonzero column.  
\begin{remark}\label{rem:latouche}\ \\
i) 
Note that Eq.~(\ref{Eq:qbdR}) and Eq.~(\ref{Eq:qbd}) imply  that the following   recursive  relation holds for all $\nu=0,\ldots,m-1$:
\begin{equation}\label{Eq:pim2}
\pi^{m}=\pi^{\nu}\prod_{k=\nu+1}^m R_{k}.
\end{equation} 

ii)
It is easy to see that the above defined $\pi^m$ and  $R_m$ satisfy the non-linear Eq.~(12.2) of 
\cite{latouche1999introduction}. 
The matrices $R_m$   are solutions to Eq.~(12.11)  of the same book, given there but 
without the explicit procedure of Eq.~(\ref{Eq:qbd})  to compute them.
\end{remark}

To obtain  the steady state distribution,  $\pi=[\pi^0,\pi^1,\ldots]$,  one only  needs to compute   
$\pi^{0}$, which per  Theorem 3 of \cite{sslqsf2015},  is given by Eqs.~(\ref{Eq:qbdresults}) - (\ref{Eq:Smqbd}) below.
 \begin{equation}\label{Eq:qbdresults}
\pi^{0}=\delta_{0}\left[S_{0}^{M_2}\delta_{0}-B^{0} \,\right]^{-1},
\end{equation}
where
\begin{equation}\label{Eq:Smqbd}
S_{0}^{M_2}= 1'_{0}+\sum_{m=1}^{M_2}\prod_{k=1}^m R_{k}1'_{m} .
\end{equation}

The procedure to calculate the steady state distribution $\pi$ when there is a down entrance state in every level that 
 is based on Eqs. (\ref{Eq:qbdresults}),  (\ref{Eq:Smqbd}) and (\ref{Eq:qbd})  above    will be referred to in the sequel as the {\sl 
 quasi birth and death down entrance state algorithm} (QDESA).

\subsection{Solution Procedures for Specific QBD processes}\label{sec:qdesa}
Unless otherwise stated in   the remainder of the paper we will  consider homogenous level processes.  Note that for these processes $B^m=B=W+\widetilde{U}$ (defined in Eq.~(\ref{Eq:qbdB})) for all $m$. 
Depending on the structure of  the  matrix $B$  we  define  two subclasses,   of decreasing generality, of the 
QDESA  procedure.
First, we identify homogenous QBD processes with a down entrance state where the matrix $B$ is of countable dimension and has the following form:
\begin{equation}\label{eq:b}
B= \left[\begin{array}{rccccc}
-b^d_0-b^u_0&b^u_0&0&0&0&\cdots\\ 
b^d_1+b^z_1&-b^w_1&b^u_1&0&0&\cdots\\ 
b^z_2&b^d_2&-b^w_2&b^u_2&0&\ddots\\
b^z_3&0&b^d_3&-b^w_3&b^u_3&\ddots\\
b^z_4&0&0& b^d_4&-b^w_4&\ddots\\
\vdots&\vdots&\ddots&\ddots&\ddots&\ddots
\end{array}\right],
\end{equation}
where
\begin{equation*}
b_{i}^{w}=b_{i}^{z}+b_{i}^{d}+b_{i}^{u},
\end{equation*}
and these elements $b_{i}^{a}$ are nonzero for $a\in\{w,z,d,u\}$. The procedure to find the steady state distribution of these processes will be referred to as \qbdp.

Second, we consider homogenous QBD processes with a down entrance state where the matrix  $B$ has the structure of Eq.~(\ref{eq:b}) and is \emph{element homogenous} i.e.,  
$$b_{i}^{a}=b^{a} \mbox{\ \ for all $i=0,1,\ldots$ and $a\in\{z,d,w,u\}.$}$$
In this case the procedure to find the steady state distribution $\pi$  will be named \qbdpp. 

In \cite{Katehakis2014inverse}  we present a fast $\mathcal{O}(\ell^{2})$ algorithm to compute the inverse of matrix $B$ of Eq.~(\ref{eq:b}),   when it is element homogenous, and thus used in  \qbdpp. In that  same paper we described a  procedure with the same complexity to compute the inverse of $B,$ when it  has the structure of Eq.~(\ref{eq:b})  and it  is not required to be element homogenous. An alternative method of computation with the same complexity is given in \cite{heinig1984algebraic}, pp. 62, but only if $\ell<\infty$ and $B$ is element homogenous.

\begin{remark}\label{rem:strucW}
One can determine which solution method is applicable by inspection of the matrix $Q$. If $W^m$ has a birth and death structure, \qbdp is applicable, and when both $W$ and $\widetilde{U}$ have a homogenous  birth and death structure, \qbdpp is applicable. 
\end{remark}

When $W$ has another structure than the one described above, it might still have a sparse form. In that case it might be beneficial to use other fast matrix inversion algorithms, like in  \cite{hager1989updating} and \cite{woodbury1950inverting}.

In the rest of this paper references to QDESA include the special cases \qbdp and \qbdpp\ as well 
and  it is assumed that the most efficient form  QDESA is always applied.

\section{Applications: Classic Queueing Models}\label{sec:examples}

In this section we will discuss two classical queueing models and analyze how the procedures above can be used to compute the steady state distribution. The Priority Queue will be discussed in detail,  and the  Longest Queue    more briefly. To avoid confusion we will use when necessary the notation $A^P$ and $A^L$ to distinguish a matrix $A$ associated with the priority model  of Section~\ref{sec:priority}, or  
the longest queue  of Section~\ref{sec:longest}, respectively. 
%
%
%
\subsection{The Priority Queue}\label{sec:priority}
In the priority queue model customers arrive according to two independent Poisson processes  with rate $\lambda_i$ for   queue $i$, $i=1,2$. There is a single  server that serves at exponential  rate $\mu,$ independently of the arrival processes. The server serves customers at   queue 2 only when queue 1  is empty, preemptions are allowed and server switches are instantaneous. 
Under these assumptions the  state of the system can be summarized by a tuple $(n, j)$ where  $n$  (respectively $j$)  is the number of customers in queue 2 (respectively  in queue 1). 

It is easy to see that $Q$ is the transition rate matrix of  a DES process, in fact a homogenous level QBD  process with $M=\infty;$  the level sets $\mathcal{L}_n$ and their entrance states $(n,0)$  are illustrated in Figure \ref{fig:pq}.

\begin{figure}[ht!]
  \centering
\includegraphics[clip=true, trim=90 500 90 100,scale=0.65]{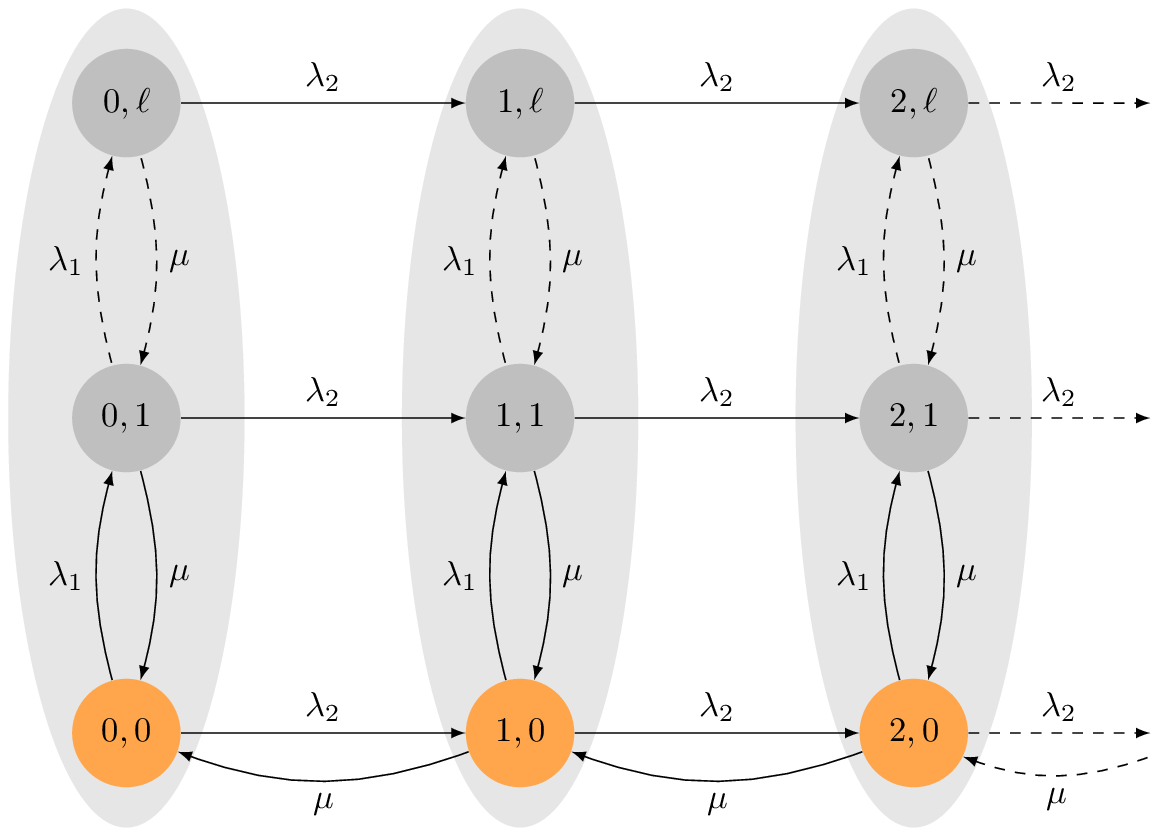}
 \caption{Transition diagram of the priority queue model.}\label{fig:pq}	
\end{figure}

Since there is no maximum for the number of customers in  queue 1 the sub-matrices $D$, $W$ and $U$ have infinite dimension ($\ell=\infty$) and the   representation below, where 
$d=(\lambda_1+\lambda_2+\mu)$.
Note that $W_{0}$ is obtained from $W$ by replacing $d$ in its $(0,0)$ position by $(\lambda_{1}+\lambda_{2})$, since in state $(0, 0)$ there are no customers in service.

\begin{equation*}
D=\left[\begin{array}{ccc} \mu&0&\cdots\\0&0&\cdots\\\vdots&\vdots&\ddots\end{array}\right], 
U=U^0\!\!=\!\!\left[\begin{array}{cccc} \lambda_2&0&0&\cdots\\0&\lambda_2&0&\ddots\\0&0&\lambda_2&\ddots\\\vdots&\ddots&\ddots&\ddots\end{array}\right],
W=\left[\begin{array}{cccc} -d&\lambda_1&0&\cdots\\ \mu&-d&\lambda_1&\ddots\\ 0&\mu&-d&\ddots\\\vdots&\ddots&\ddots&\ddots\end{array}\right],
\end{equation*}

Note that in this model we have:  $U^0=U=\lambda_2 I,$ thus,
$R^{P}=R_1^{P}:= -\lambda_2 B^{-1},$ where 
$$
B^{P}=\left[\begin{array}{ccccc} -(\lambda_1+\mu)&\lambda_1&0&0&\cdots\\ 
\lambda_2+\mu&-d&\lambda_1&0&\cdots\\ \lambda_2&\mu&-d&\lambda_1&\ddots\\\lambda_2&0&\mu&-d&\ddots\\
\vdots&\vdots&\ddots&\ddots&\ddots\end{array}\right] .
$$

It is clear that matrix $B^{P}$ has the required structure to use the \qbdpp .   Thus, the priority  queue model can be solved easily using this method.

\subsection{Longest Queue}\label{sec:longest}

In a longest queue model,  cf.~\cite{zheng1990queueing},    two types of  customers arrive according to
 independent Poisson   streams, each with rate $\lambda$
and form two queues according to their type.  
There is  a single   exponential server with rate $\mu> 2\lambda$ that severs  customers  from the longest queue (\ie the one having the most customers),  where ties are resolved with equal probabilities for each queue; server queue switches are instantaneous. 

To obtain meaningful results for this model, we will use
the following state space description that is easy to work with. At each point of time let the 
state be specified by a tuple  $(n,j)$, where $j$ denotes the difference 
between the two queue lengths  and $n$ denotes  the  length of the shortest queue. A more natural state space  description is discussed in Section~\ref{sec:longest2}.

It is easy to deduce that this  is a DES process, in fact a homogenous level QBD  process,  with $M=\infty$  with level sets $\mathcal{L}_n$  as described in Section \ref{sec:def} and entrance states $(n,1)$ for level $n$ where matrices   $D$, $U$, $W$ as given below, $d=2\lambda +\mu .$
We note that $W_{0}$ is obtained from $W$ by replacing $d$ in its $(0,0)$ position by $(\lambda_{1}+\lambda_{2})$, since in state $(0, 0)$ there are no customers in service.

\begin{equation*}
D=\left[\begin{array}{cccc} 0&\mu&0&\cdots\\0&0&0&\cdots\\0&0&0&\cdots\\\vdots&\vdots&\vdots&\ddots\end{array}\right],
U=\left[\begin{array}{cccc} 0&0&0&\cdots\\\lambda&0&0&\ddots\\0&\lambda&0&\ddots\\\vdots&\ddots&\ddots&\ddots\end{array}\right],
W=\left[\begin{array}{cccc} -d&2\lambda&0&\cdots\\ \mu&-d&\lambda&\ddots\\ 0&\mu&-d&\ddots\\\vdots&\ddots&\ddots&\ddots\end{array}\right].
\end{equation*}

Since $U^0=U$, the rate matrices $R_1$ and $R$ for this model are equal, \ie$R_{1}^{L} =R^{L}$, as in the previous  models and the matrix $B$ in this model has the following form:

$$
B^L=\left[\begin{array}{ccccc} -d&2\lambda&0&0&\cdots\\ 
\mu&-(\mu+\lambda)&\lambda&0&\cdots\\ 0&\mu+\lambda&-d&\lambda&\ddots\\0&\lambda&\mu&-d&\ddots\\0&\lambda&0&\mu&\ddots\\
\vdots&\vdots&\ddots&\ddots&\ddots\end{array}\right] .
$$

Note that the matrix $B^{L}$ has a structure similar (but not identical) to that of $B$ defined in Eq.~(\ref{eq:b}); its structure from the second column on  
is  identical to that of $B$, but an extra column has been added in front. This can be easily resolved with a suitable modification of \qbdpp. 

\begin{remark}
The Feedback queue, the third model that is discussed in \cite{leeuwaarden2009quasi}, fits the QDESA framework as well; its analysis goes analogous to the analysis of the priority queue.
\end{remark}
 
\section{Lattice Path Counting}\label{sec:lpc} 
A dif{}ferent  approach to 
compute     the steady state distribution $\pi$ for a class of Markov process that includes the queueing 
models described before,  is the  \emph{Lattice Path Counting   Algorithm} (LPCA)    of   \cite{van2006quasi}, see also \cite{leeuwaarden2009quasi}. In this section we will repeat LPCA in the notation used in this paper.

Throughout this paper we use a labeling of states that is  
 consistent with  our notation introduced  in  \cite{ssl2012} and \cite{sslqsf2015}. 
 In \cite{leeuwaarden2009quasi} a similar tuple notation was used, but the meaning of the first and the second element is reversed.
   For example, in the priority queue model of Section~3.1 
 we denote a system with two queues with $n$ customers in queue 2 and $i$ in queue 1 as 
 $(n,i)$.  This same $(n,i)$ in \cite{leeuwaarden2009quasi} denoted
  a system with two queues with $n$ customers in queue 1 and $i$ customers in queue 2. 

Recall that we   used the \emph{level} (first coordinate) sets 
$\mathcal{L}_n=\{(n,i),\ i=1,\ldots,\ell \}$ where $n=0,1,\ldots$ to define a partition with respect to which the studied processes are `level QBD' processes.
A `stage QBD' process can be defined analogously;
one can  rearrange the states of $\ssp$ in the order of  stages (second coordinate), 
\ie as $(0,1),\ldots, (M,1),$
 $ (0,2),\ldots (M,2),$ $\ldots,$ $ (0,\ell),\ldots ,$ $ (M,\ell).$ In this case we define the stage sets to be: $\mathcal{K}_i =\{(n,i),\ n=0,1,\ldots  \}.$ Transitions are allowed one stage up and one stage down to preserve the QBD property in the direction of stages.
Using a  stage partition,  
we obtain  the following representation of the transition generator matrix, which will be denoted by $\stQ$ to indicate that  a stage   partition is  used:
$$
\stQ = \left[\begin{array}{cccc} B_{1}&B_{0}&0&\cdots\\A_{2}&A_{1}&A_{0}&\ddots\\0&A_{2}&A_{1}&\ddots\\\vdots&\ddots&\ddots&\ddots\end{array}\right],
$$
where the dimension of the above sub matrices  is $M\times M$.

The matrix  $\stQ$  in the current paper is the same  as the matrix $Q$ of \cite{leeuwaarden2009quasi}, subject to appropriate 
  relabeling of states, as is mentioned above.  Note that in this paper 
  the notation $M$ is used for our $\ell$ above and their
corresponding $\ell$ is infinite.  

Following  the approach  introduced in \cite{leeuwaarden2009quasi}, a process $X(t)$  is called Lattice Path Countable (LPC) if the following three conditions hold:
\begin{itemize}
\item[i)] When $j>1$,  the only   transitions  allowed from state $(n,j)$ are   to states: $(n+e_1,j+e_2)\in \ssp$ where $e_1\in\{0,1\}$ and $e_2\in\{-1,0,1\}$;
\item[ii)]  When $j>1$, the transition rate $\stQ((n,j)\, ,\, (n+e_1,j+e_2) )$ is a function of the jump size and direction only, \ie
\begin{equation}\label{eq:lpcah}
 \stQ((n,j)\, ,\, (n+e_1,j+e_2) ) =\stq(e_1,e_2);
\end{equation}
 
\item[iii)]  The process is a stage QBD process where $\ell$ is infinite and $M$ is  finite or  infinite. 
\end{itemize}

In the previous section we described a rate matrix $R$ that provides a relationship between the steady state distributions of the different levels. A similar recursion can be defined for the steady state vectors $\pi_{i}$ for stage $i>0$: $\pi_{i+1}=\pi_{i}\stR,$

where $\stR$ is the minimal nonnegative solution to the matrix quadratic equation:
$A_0+\stR A_1+\stR^{2}A_2=0.$
 
We have denoted the rate matrix constructed with LPC as $\stR$ to distinguish it from the matrix $R$  used in \eqqref{Eq:qbd} above.  

Figure \ref{fig:levelstages1} displays a simplification of a transition diagram of a process that is a QBD process with respect both to the levels and to the stages. The LPCA can be applied with respect to the stages.

\begin{figure}[ht!]		
  \centering
\includegraphics[clip=true, trim=90 500 90 100,scale=0.67]{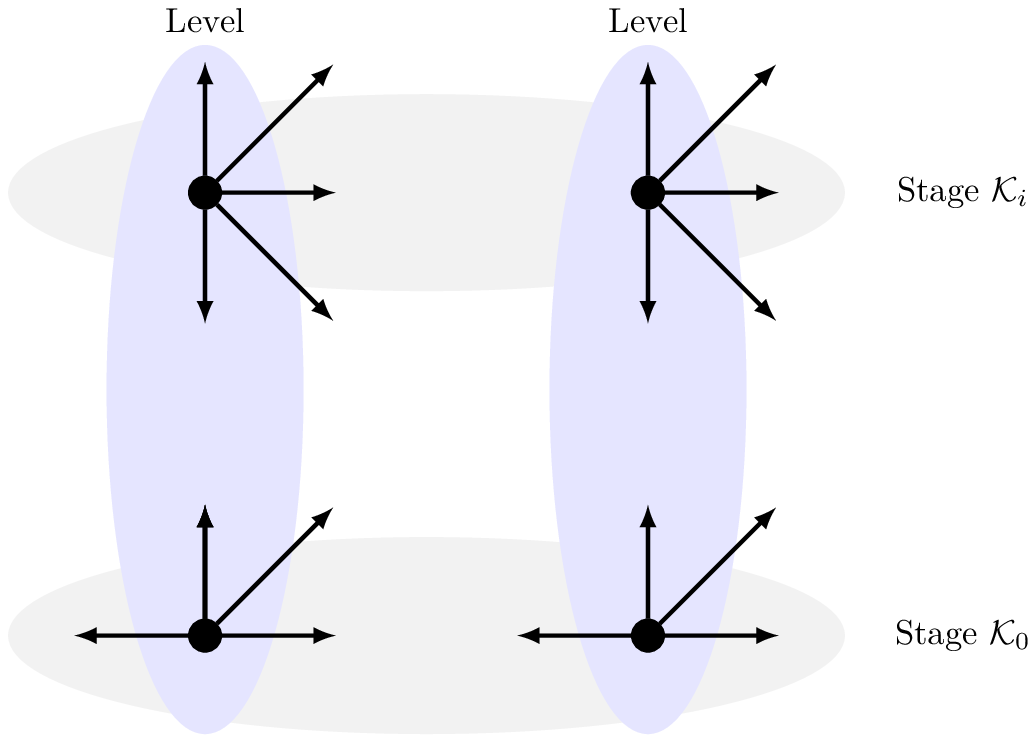}
 \caption{Levels and Stages.}\label{fig:levelstages1}
\end{figure}

Further, it is  known, cf.~for example \cite{latouche1999introduction}, that 
the elements   $\str(n| m)$    of the matrix  $\stR=[\str(n| m)] $ represent 
 the expected taboo sojourn time in $(n , i+ 1)$  before the first return to stage $i$   given that the process starts in $(m, i)$ multiplied by  the sojourn time in stage $i$, for any $i\ge 1$.
Since the LPC assumption above does not allow transitions in the  downward direction and has a homogenous structure by point ii) above, the rate matrix is  upper-triangular and
has the following form:
$$
\stR=\left[\begin{array}{cccc} \str_{0}&\str_{1}&\str_{2}&\cdots\\0&\str_{0}&\str_{1}&\cdots\\0&0&\str_{0}&\cdots\\ \vdots&\vdots&\ddots&\ddots\end{array}\right].
$$
 Theorem \ref{theo:lpc} below 
  provides  an explicit  expression for the elements 
  of   $\stR$.  It is the main result of  \cite{leeuwaarden2009quasi}  and uses the following expressions:
%
%
 \begin{align}
P_{h}(s,u,m)&=\phi\langle1,-1\rangle^{s}\phi\langle1,0\rangle^{t}\phi\langle1,1                                                                                                                                                                                                                                                                                                                                                                                                                                                                                                                                                                                                                                                                                                      \rangle^{u}\phi\langle0,1\rangle^{m-u}\phi\langle0,-1\rangle^{m+1-s} \nonumber \\
L_{h}(s,u,m)&=\frac{1}{m+1}\binom{2m}{m}\binom{m+1}{s}\binom{m}{u}\binom{2m+t}{t} \nonumber \\
\label{eq:G} 
G_{h}&=\sum_{s=0}^{h}\sum_{u=0}^{h-s}\sum_{m=\max(u,s-1)}^{\infty}L_{h}(s,u,m)P_{h}(s,u,m)\\
\kappa_{h}&=\frac{\philr{1,0}\kappa_{h-1}+\philr{0,1}\sum_{j=0}^{h-1}G_{h-j}\kappa_{j}+\philr{1,1}\sum_{j=0}^{h-1}G_{h-j-1}\kappa_{j}}{1-\philr{0,1} G_{0}}, \nonumber 
\end{align}
where $\rho_{0}=1$ and $\rho_{-1}=0 $ and $\phi(e_1,e_2)$  denotes the    transition  probability
from state $(n, j)$ to state $(n + e_1, j + e_2).$

\begin{theorem}\label{theo:lpc} The upper diagonal elements $\str_{h}$ of $\stR$ can be expressed as follows:
\begin{equation}\label{eq:lpc}
\str_{h}=2\, \frac{\philr{0,1}\kappa_{h}+\philr{1,1} \kappa_{h-1}}{1+\sqrt{1-4\philr{0,1} \philr{0,-1}}}.
\end{equation}
\end{theorem} 

The LPCA is based on the calculation of \eqqref{eq:lpc}, utilizing a new  computation of the $G_h$ in Eq.~(\ref{eq:G}) above using hypergeometric functions, cf. Eq.~(26) and (27) of  \cite{leeuwaarden2009quasi}.

\section{Comparative Analysis}\label{sec:comp}
In this section we will compare the efficiency of LPCA and QDESA described in the previous section. To make a fair comparison between these algorithms we will compare their complexities  in Section \ref{sec:complexity} for transition rate matrices on which they 
can \emph{both} be applied. In Section~\ref{sec:app} we discuss classes of models for which a version of QDESA is applicable while LPCA is not. We will also distinguish structures for which the LPCA can be used efficiently, but for which QDESA is not readily applicable. 

It is important to note that LPCA is based on the existence of  a  `homogeneous portion'  of stages,  \ie transition rates are both stage and level independent, as is described in Section 5 of \cite{leeuwaarden2009quasi} and summarized in the previous section.
The non-homogeneous part of the state space is considered to be (part of) stage  $\mathcal{K}_0$. This non-homogeneous part may induce that QDESA might not be applicable; the entrance state property might be violated. Exit states might still be present, for the formal definition of an exit state we refer to \cite{Katehakis2014LevelProductForm}. In this paper we have described how an entrance state and an exit state are related and how the choice of levels can be adjusted to transform an exit state into an entrance state. However, no  applications are known for which such a complex structure in $\mathcal{K}_{0}$ is necessary, that QDESA is no longer applicable.  

When a process has such a structure that  QDESA applies (with respect to the levels) \emph{and} LPCA (with respect to the stages) we note that $B,$ (where  $R=UB^{-1}$) has to have the structure of Eq.~(\ref{eq:b}), up to a permutation of the columns, due to the fact that the process is a QBD process in the stage direction, see Remark~\ref{rem:strucW}.
Furthermore, it is easy to see that  this homogeneous structured process implies that  matrix $B$  has an element homogenous structure, since the elements are independent on the stages.
Summarizing the above, we state the following.
\begin{proposition}
Suppose that the following are both true:
\begin{itemize}
\item[-] LPCA is applicable to a QBD process with respect to the stages,
\item[-] The set $\bigcup_{k=0}^{n} \mathcal{L}_{k}$ has an entrance state or the set $\bigcup_{k=n}^{M} \mathcal{L}_{k}$ has an exit state.
\end{itemize}
Then \qbdpp\ can be applied with respect to the level partition.

\end{proposition}
A result of this proposition is that for a fair computational comparison between the algorithms it suffices to  compare LPCA with \qbdpp.

\subsection{Computational Complexity of the Procedures}\label{sec:complexity}
By  Eq.~(\ref{Eq:qbd}) we know that the computational complexity of \qbdpp is determined by the 
complexity of calculating the elements of the matrix $R$ with dimension $\ell\times \ell$. Since $U$ is a sparse matrix in this case, the computationally heavy step is to invert matrix $B$.
For LPCA  the computational complexity is determined by the 
complexity of calculating the elements of matrix $\stR$. Recall that $\stR$ has dimension $M\times M$.

The  general result on complexity is 
summarized in Theorem \ref{th:cplx} below. To compare the complexities of QDESA to that of LPCA, we take $\ell =M$, e.g.~this is the case in the priority queue model 
when the queues have the same (finite or truncated) capacity. In the following complexity analysis we assume that arithmetic operations with individual elements have complexity $\mathcal{O}(1).$

\begin{theorem}\label{th:cplx}  When the steady state distribution of a QBD process can be found both by using LPCA and using QDESA the following are true:

i) Using LPCA, the computation of the stage-rate matrix $\stR$ has complexity $\mathcal{O}(M^{4}).$

ii) Using \qbdpp, the computation of the level-rate matrix $R$ has complexity $\mathcal{O}(\ell^{2}).$
\end{theorem}
\begin{proof}
To prove part \emph{i}) we assign  complexity of  $\mathcal{O}(h)$ to the computation of the term  
$\sum_{m=\max(u,s-1)}^\infty L_h(s,u,m)P_h(s,u,m)$  that involves hypergeometric functions, 
cf. Eq.~(26) and Eq.~(27)  of \cite{leeuwaarden2009quasi},  noting  that $s+u+t=h$.
The \emph{correct} complexity of the above computation is actually higher, but this lower bound  is easy to establish when counting conservatively.
From \eqqref{eq:G} we see that to calculate  $G_h$ we need approximately 
 $(h^2/2) \mathcal{O}(h)=\mathcal{O}(h^3)$ 
iterations (a double summation). The computation of matrix $\stR$ (of size $M\times M$)  requires 
the computation of  all its $M$ different nonzero elements,  
$\str_{0},\ldots,\str_{M-1}$ and each of these computations is of complexity  $ \mathcal{O}(h^3)$. The complexity   
  of the computation of rate matrix $\stR$ is: 
 $\sum_{h=0}^{M-1}\mathcal{O}(h^3)=\mathcal{O}(M^4)$.

For part \emph{ii}), we will establish the complexity for the \qbdpp. The procedure for the computations of the elements of the first row and first column of $C$ uses a single computation per element, of $\mathcal{O}(1)$. For the remaining elements a linear expression has to be solved, having a complexity of $\mathcal{O}(1)$ per element as well. Thus the total complexity of 
computing $C$  is $\mathcal{O}(\ell^2)$, the number of elements of $B^{-1}$. The matrices $U$ have a sparse form (at most 3 non-zero elements per row), induced by  the fact that LPCA is applicable by assumption. Since $R=UB^{-1},$ the complexity of computing $R$ is $\mathcal{O}(\ell^2)$: both the complexity of the matrix multiplication $UB^{-1}$ and of the calculation of $B^{-1}$ have this complexity.
The proof is complete.
 \end{proof}

\begin{remark}
For some special cases, e.g.~the priority queue, the complexity of LPCA is lower because 
of the absence of transitions from $(n,j)$ to $(n+e_{1},j+e_{2})$ with $(e_1,e_2)\in \{\langle -1,1 \rangle ,  \langle 1,1 \rangle \}$ for all $(n,j)$.
In this special case the complexity of LPCA is $\mathcal{O}(M^2),$ because in the computation of $G_h$,  both $s=0$ and 
 $u=0$ and the summation in Eq.~(\ref{eq:G}) is only over  $m;$ i.e., the complexities of LPCA and QDESA are the same in this case.
 \end{remark}

\begin{remark}
When there is no additional structure on matrix $B,$ both \qbdp and \qbdpp can not be used, so we need a general matrix inversion to compute $B^{-1}$ of dimension $\ell$ by $\ell$ 
that is in complexity less than $\mathcal{O}(\ell^{2.379})$, cf.~\cite{williams2012multiplying}, when $\ell$ is finite. When $U$ is a non-sparse matrix this provides a solution procedure with total complexity $\mathcal{O}(l^{3})$ for QDESA.
\end{remark}

\section{The Applicability of QDESA to More General Models}\label{sec:app}
In this section we will determine the differences in applicability between QDESA and LPCA, and display these differences with examples. We will consider variations of the queues in Section~\ref{sec:priority} and \ref{sec:longest} that can be solved with QDESA but not with LPCA. 

One of the main advantages of QDESA over LPCA is that QDESA not only provides a method to find the rate matrix, but  the algorithm includes a way to find the steady state distribution using this rate matrix. Since LPCA does not require any restrictions on the non-homogenous part $\mathcal{K}_{0},$ the structure on this set can be very complex and a direct technique to do this step is absent and not trivial to include. Therefore QDESA can be viewed as a more complete solution procedure.
And for that reason we will not discuss models that have a complicated structure on $\mathcal{K}_{0}$; even though it is possible to find the rate matrix for such a model with LPCA, but perhaps not with QDESA, within the LPCA no procedure is provided to find the steady state distribution.  

There are four important classes of models for which (an extension of) QDESA is applicable and for which the 
LPCA can not be used at all.  The first class involves element non-homogenous DES processes: in this case there is no homogeneous tail on which the LPCA is applicable.  The second class involves 
processes with a finite number of stages $\ell$, as described in Section \ref{sec:def};  in the LPC case
there is analysis only for the case in which the number of stages $\ell$ is infinite.  The third class involves 
DES processes with  `down' transitions  to the entrance state in a level $L_{m-1}$  from more than one state in level $L_{m}$ for some $m$. 
The fourth and most general  class involve all DES processes, \ie Markov chains with 
 transitions   from an arbitrary state $(n,j)$  
to states: $(n+e_1,j+e_2)\in \ssp$ where $e_1\in\{0,1,\ldots\}$ and $e_2\in\{\ldots,-1,0,1,\ldots\}$, under the condition of a single entrance state in the `down' direction cf. \cite{sslqsf2015}.

Conversely, there are  processes for which the LPCA is applicable, but QDESA is not. Such processes will contain  transitions that destroy the DES property with respect to the level partition. For example transitions from a state $(n,1)$ to $(n-2,1)$  are allowed in an LPC Process, but are not allowed in  a DES process, when $(n,1)$ is the entrance state for every level $\mathcal{L}_{n}$. 
However, by relabeling and changing the levels one can construct a DES process in a lot of cases.

Table~\ref{table:LPCA-SLP} identifies the difference in applicability between the two procedures. We note that the transitions within the heterogenous stage $\mathcal{K}_{0}$ are not restricted, i.e.~matrix $B_{0}$ and $B_{1}$  are possibly non-sparse matrices in the LPCA procedure. We compare this with the restrictions that are imposed by QDESA.
\vspace{0.2cm}

\begin{table}[ht!]
{\small 
\begin{tabular*}{\textwidth}{@{\extracolsep{\fill}}|p{6cm}|p{6cm}|}
\hline
  \multicolumn{2}{|c|}{\textbf{Stage $\mathcal{K}_{0}$}, the Non-Homogeneous portion} \\
  \hline
  \textbf{LPCA}&\textbf{QDESA}\\
  \hline
  Within this stage all transitions allowed.&QSF Structure should be obeyed.\\
  Transitions leaving $\mathcal{K}_{0}$ allowed only to $\mathcal{K}_{1}$.&Transitions are allowed to all higher stages.\\
 Element Non-Homogeneous.&Element Non-Homogeneous.\\
 Sol.~Proc.~on $\mathcal{K}_{0}$ not included in algorithm.& Solution procedure included for all levels.\\
\hline
\multicolumn{2}{c}{\vspace{-0.3cm}}\\
\hline
  \multicolumn{2}{|c|}{\textbf{Stage $\mathcal{K}_{i}$} from the Homogeneous portion} \\
  \hline
   \textbf{LPCA}&\textbf{QDESA}\\
  \hline
  Nearest Neighbor structure within levels.&All transitions allowed within levels.\\
Nearest Neighbor to `NE', `E', `SE'.&All transitions allowed to higher levels.\\
Element \emph{Homogeneous}.&Element Non-Homogeneous.\\
No transitions to `NW', `W', `SW' allowed.&Trans.~to `W' allowed to \emph{entrance} state.\\
Number of stages must be infinite.&Number of stages can be finite or infinite.\\
\hline
\end{tabular*}
\caption{Restrictions   for the applicability of LPCA  and QDESA.}\label{table:LPCA-SLP}
}
\end{table}

\subsection{The Priority Queue with Batch Arrivals}\label{sec:priority2}
Consider  the priority queue model where two types  of customers arrive in  batches according to independent Poisson processes  with rate $\lambda_i$ for   queue $i$, $i=1,2$. Upon arrival the size $Z_i$ of a batch of type $i$  becomes known. For each fixed $i$ 
 the $Z_i$ are iid random variables that follow a known discrete distribution: 
  $P(Z_i=z)=p_i(z).$  

There is a single  server that serves at exponential  rate $\mu,$ independent of the arrival processes. The server serves customers at   queue 2 only when queue 1  is empty, preemptions are allowed and switches are instantaneous. 
Under these assumptions the  state of the system can be summarized by a tuple $(n, j)$ where  $n$  (respectively $j$)  is the number of customers in queue 2 (respectively  in queue 1). Because we assume that there is no maximum for number of customers in  queue 1   the sub-matrices of $Q$ have infinite dimension. It is easy to see that $Q$ is the transition rate matrix of  a successively lumpable process with respect to the levels with $M_1=0,$ $M_2=\infty$ and the following within- and up-matrices, where 
$d=(\lambda_1+\lambda_2+\mu)$:
\small
$$
 W =\left[\begin{array}{cccc} -d&\lambda_1p_1(1)&\lambda_1p_1(2)&\cdots\\ \mu&-d&\lambda_1p_1(1)&\ddots\\ 0&\mu&-d&\ddots\\\vdots&\ddots&\ddots&\ddots\end{array}\right],
 \quad
 U^{nk} =\left[\begin{array}{cccc} \lambda_2p_2(k)&0&0&\cdots\\0&\lambda_2p_2(k)&0&\ddots\\0&0&\lambda_2p_2(k)&\ddots\\\vdots&\ddots&\ddots&\ddots\end{array}\right].
$$ 

The matrix $W^{0}$ has its $(1,1)$ element equal to $-(\lambda_1+\lambda_2)$ and all its other elements are 
the same as those of $W$.  The matrix $D$ is the same as that of the process described 
in Section 3.1.
This model can be solved using QDESA, but LPCA is not applicable.

\subsection{Longest Queue Model with non-homogeneous arrival rates}\label{sec:longest2}
We will extend the model discussed in Section \ref{sec:longest} in such a way that  now two types of  customers arrive according to
 independent Poisson   streams, with rate $\lambda_{1}$ and $\lambda_{2}$. There is  a single   exponential server with rate $\mu> \lambda_{1}+\lambda_{2}$.  Note that the fact that the arrivals have a different rate implies that the state space description used in Section~\ref{sec:longest}  does not induce a Markov chain.
Therefore, we now let the state be specified by a tuple $(n,j)$ where $j$ denotes the number of customers in queue 1 and $n$ the number of customers in queue 2. The buffers are of size $M$ and $\ell$ respectively and can be either finite of infinite.
The transition diagram is displayed in Figure~\ref{fig:longestqueue2} and the level partition is highlighted by the grey background.
It is easy to deduct that this  is a DES process where the level sets $\mathcal{L}$ are formally described as follows: 
$$\mathcal{L}_{m}=\bigcup_{n=m}^{M}\{(n,m-1)\}\cup\bigcup_{i=m}^{\ell}\{(m-1,\ell)\}\cup\{(m,m)\}.$$
State $(m,m)$ is the entrance states for the set $\bigcup^{m}_{k=0} \mathcal{L}_{k}.$
With this different arrival rates, LPCA can not be used, while   \qbdp \ can be used. Note that the rate matrix $R_{m}$ depends on the level $m$.

\begin{figure}[ht!]
\begin{center}
\includegraphics[clip=true, trim=100 420 100 100,scale=0.67]{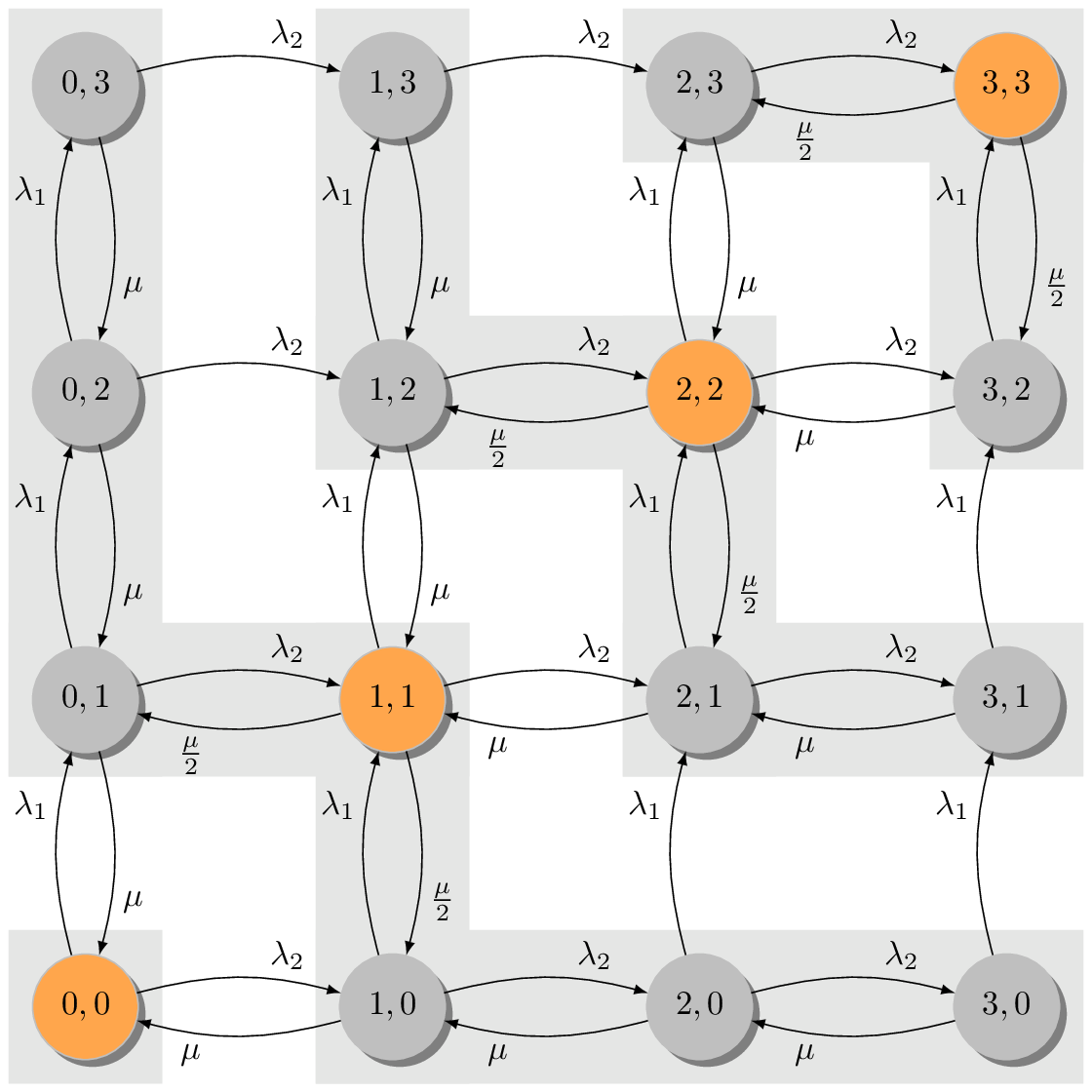}
\caption{Longest Queue model.}\label{fig:longestqueue2}
 \end{center}
 \end{figure}

\acks
This Research has been partially supported by the 
National Science Foundation with grant  CMMI-14-50743.

\bibliographystyle{apt} 
\bibliography{LPCref}

\end{document}